\documentclass[reqno]{amsart}

\usepackage{graphicx}
\usepackage{mathrsfs}
\usepackage{enumerate}
\usepackage{amsmath}
\usepackage{amsfonts}
\usepackage{latexsym}
\usepackage{wrapfig}
\usepackage{float}
\usepackage{verbatim}
\usepackage{hyperref}
\usepackage{url}
\usepackage{bbold}

\usepackage{pgf,tikz,pgfplots}
\pgfplotsset{compat=1.14}
\usepackage{mathrsfs}
\usetikzlibrary{arrows}
\usepackage[export]{adjustbox}

\definecolor{grey}{rgb}{0.8,0.8,0.8}

\newcommand{\A}{\mathscr{A}}
\newcommand{\B}{\mathcal{B}}
\newcommand{\C}{\mathcal{C}}
\newcommand{\F}{\mathcal{F}}
\newcommand{\N}{\mathbb{N}}
\newcommand{\Q}{\mathbb{Q}}
\newcommand{\Z}{\mathbb{Z}}
\newcommand{\X}{\mathscr{X}}
\newcommand{\Y}{\mathscr{Y}}

\newcommand{\bu}{\mathbf{u}}
\newcommand{\bv}{\mathbf{v}}
\newcommand{\bw}{\mathbf{w}}
\newcommand{\bx}{\mathbf{x}}
\newcommand{\by}{\mathbf{y}}

\newcommand{\g}{\gamma}

\newcommand{\Brac}{\text{Brac}}

\newcommand{\df}[1]{\emph{#1}}

\newtheorem{theorem}{Theorem}[section]
\newtheorem*{theorem*}{Theorem}
\newtheorem{lemma}[theorem]{Lemma}
\newtheorem{proposition}[theorem]{Proposition}
\newtheorem{crl}[theorem]{Corollary}

\theoremstyle{definition}
\newtheorem{definition}[theorem]{Definition}
\newtheorem{example}[theorem]{Example}
\newtheorem{conjecture}[theorem]{Conjecture}
\newtheorem*{conjecture*}{Conjecture}

\theoremstyle{remark}
\newtheorem{remark}[theorem]{Remark}

\numberwithin{equation}{section}

\usepackage{pgf,tikz,pgfplots}
\pgfplotsset{compat=1.14}
\usepackage{mathrsfs}
\usetikzlibrary{arrows}

\begin{document}

\title[Unistructurality of cluster algebras from surfaces]{Unistructurality of cluster algebras from surfaces without punctures}

\author{V\'eronique Bazier-Matte}
\address{Laboratoire de combinatoire et d'informatique math\'ematique, Universit\'e du Qu\'ebec \`a Montr\'eal, CP 8888, Succ. Centre-ville, Montr\'eal (Qu\'ebec), H3C 3P8, Canada}
\email{bazier-matte.veronique@courrier.uqam.ca}

\author{Pierre-Guy Plamondon}
\address{Laboratoire de Math\'ematiques d'Orsay, Universit\'e Paris-Sud, CNRS, Universit\'e Paris-Saclay, 91405 Orsay, France}
\email{pierre-guy.plamondon@math.u-psud.fr}

\thanks{This project was initiated during a research stay of the first author at Universit\'e Paris-Sud XI, with the financial support of schorlarships from CRM and Mitacs Globalink for international internship.  The first author is supported by the Alexander Graham Bell Canada Graduate Scholarships-Doctoral from NSERC.
The second author is supported by the French ANR grant SC3A (ANR-15-CE40-0004-01) and by a PEPS ``Jeune chercheuse, jeune chercheur'' grant.}  

\definecolor{grey}{rgb}{0.75,0.75,0.75}

\begin{abstract}
A cluster algebra is unistructural if the set of its cluster variables determines its clusters and seeds.  It is conjectured that all cluster algebras are unistructural. In this paper, we show that any cluster algebra arising from a triangulation of a marked surface without punctures is unistructural.  Our proof relies on the existence of a positive basis known as the bracelet basis, and on the skein relations.  We also prove that a cluster algebra defined from a disjoint union of quivers is unistructural if and only if the cluster algebras defined from the connected components of the quiver are unistructural.
\end{abstract}

%%%%%%%%%%%%%%%%%%%%%%%%%%

\maketitle

\tableofcontents

%%%%%%%%%%%%%%%%%%%%%%%%%%

\section{Introduction}
Cluster algebras are commutative algebras that come equipped with a rich structure: 
their generators, the \emph{cluster variables}, are grouped into finite sets called \emph{clusters}, each of which is endowed with a matrix (or a quiver) to form a \emph{seed},
in such a way that all seeds can be deduced from any one of them by a process called \emph{mutation}.
Ever since their introduction and fundamental study by S.~Fomin and A.~Zelevinsky \cite{FZ02, FZ03, FZ07} (and also with A.~Berenstein in \cite{BFZ05}), 
cluster algebras have received a lot of attention due to their appearance in many fields of mathematics.
Survey papers and books on cluster algebras include \cite{Marsh13, LW14, W14, GR17}; let us also mention \cite{GSV10}, which focuses on links with Poisson geometry, \cite{Reiten10, Amiot11, Keller12, Plamondon18} on links with representation theory of associative algebras, \cite{GLS13} on links with Lie theory, and \cite{Schiffler18} on links with triangulations of surfaces.

In this paper, we are interested in structural questions for cluster algebras.  
If one forgets the cluster algebra structure given by the seeds, many cluster algebras become isomorphic as rings.
In \cite{ASS12, ADS14}, a suitable notion of morphism of cluster algebras was introduced: roughly speaking, such a morphism should send seeds into seeds.
An isomorphism of cluster algebra is then an isomorphism of rings satisfying the strong condition that it should preserve seeds.

In \cite{ASS14} was introduced the notion of unistructurality: a cluster algebra is unistructural if the set of its cluster variables determines its seeds (see Section~\ref{sect::prelim}, where we discuss several variants of the definition).  %There is, to this day, no example of a cluster algebra which is known not to be unistructural. 

\begin{conjecture*}[1.2 of \cite{ASS14}]
Any cluster algebra is unistructural.
\end{conjecture*}

The conjecture was proved in \cite{ASS14} for skew-symmetric cluster algebras of Dynkin type or of rank~$2$, and in~\cite{BM16} for cluster algebras of type~$\widetilde{\mathbb{A}}_n$ (all of these cluster algebras are taken with trivial coefficients).
In the latter case, the first author of the present paper used the geometric model of~\cite{FST08} (see Section~\ref{sect::surfaces}), in which cluster algebras of type~$\widetilde{\mathbb{A}}_n$ are linked to triangulations of an annulus.  In this paper, we extend these methods to any unpunctured surface.

\begin{theorem*}[\ref{theo::main}]
    Any cluster algebra with trivial coefficients associated to a triangulation of an unpunctured surface is unistructural.
\end{theorem*}

The proof, given in Section~\ref{sect::proof}, relies on the existence of a positive basis (the \emph{bracelet basis} of \cite{MSW13}) and on the skein relations to analyze products of cluster variables.  On the way, we obtain a reduction result (Proposition~\ref{prop::connected}) showing that it is sufficient to prove the unistructurality conjecture for cluster algebras defined from connected quivers.

Shortly after this paper appeared, the conjecture was proved in full generality by P.~Cao and F.~Li \cite{CL18}.

%%%%%%%%%%%%%%%%%%%%%%%%%%

\section{Cluster algebras and unistructurality}\label{sect::prelim}

\subsection{Cluster algebras}

Let us recall the main definitions and properties of cluster algebras.  We will mainly follow the exposition of~\cite{FZ07}, with the additional restrictions that we will only consider skew-symmetric cluster algebras (which allows us to use quivers instead of matrices) with trivial coefficients.
Let us mention that examples of all notions recalled in this section can be easily generated using B.~Keller's applet \cite{KellerApp} or Sage \cite{Sage}.

Our first ingredient is that of quiver mutation.  A quiver is an oriented graph (possibly with multiple edges); if~$Q$ is a quiver, we denote by~$Q_0$ its set of vertices,~$Q_1$ its set of arrows, and by~$s,t:Q_1\to Q_0$ the maps sending each arrow to its source or its target, respectively.

\begin{definition}
  Let~$Q$ be a finite quiver without oriented cycles of length~$1$ or~$2$, and let~$i$ be a vertex of~$Q$.
  The \emph{mutation of~$Q$ at~$i$} is the quiver~$\mu_i(Q)$ defined by modifying~$Q$ as follows:
  \begin{enumerate}
      \item for each path of length~$2$ of the form~$h\xrightarrow{\alpha} i \xrightarrow{\beta} j$ (recall that~$i$ is fixed), add an arrow~$h\xrightarrow{[\beta\alpha]} j$;
      \item reverse the orientation of each arrow having~$i$ as source or target;
      \item remove the arrows from a maximal set of disjoint oriented cycles of length~$2$ created in the two previous steps.
  \end{enumerate}
\end{definition}

It is immediate from the definition that mutation of quivers at a given vertex is an involution.  

\begin{definition}
  A \emph{seed} is a pair~$(\bu, Q)$, where
  \begin{itemize}
      \item $Q$ is a finite quiver without oriented cycles of length~$1$ or~$2$, and with vertex set~$Q_0 = \{1, \ldots, n\}$;
      \item $\bu = (u_1, \ldots, u_n)$ is a free generating set of the field~$\Q(x_1, \ldots, x_n)$.  Note that the variables in~$\bu$ are indexed by the vertices of~$Q_0$.
  \end{itemize}
  
  The \emph{mutation of~$(\bu, Q)$ at~$i$} is the pair~$\mu_i(\bu, Q) = (\bu', Q')$, where
  \begin{itemize}
      \item $Q' = \mu_i(Q)$ is the mutation of~$Q$ at~$i$;
      \item $u'_{\ell} = u_{\ell}$ if~$\ell\neq i$, and~$u'_i$ satisfies the \emph{exchange relation}
                   \[
                   u'_i = \frac{\prod_{\stackrel{\alpha\in Q_1}{t(\alpha) = i}} u_{s(\alpha)} + \prod_{\stackrel{\alpha\in Q_1}{s(\alpha) = i}} u_{t(\alpha)}}{u_i}. 
                   \]
  \end{itemize}
  Two seeds are \emph{mutation-equivalent} if one can be obtained from the other by a sequence of mutations.
  Two seeds are \emph{isomorphic} if they are the same up to reordering of their cluster variables and relabelling of the correponding vertices of their quivers.
\end{definition}

One readily checks that the mutation of a seed is still a seed, and that mutation of seeds at a given vertex is an involution.

\begin{definition}
    Let~$(\bu, Q)$ be a seed.  Consider all seeds~$(\bu', Q')$ obtained by iterated mutations of~$(\bu, Q)$. 
    \begin{itemize}
        \item The~$n$-tuples~$\bu'$ appearing in those seeds are called \emph{clusters}.
        \item The elements of clusters are called \emph{cluster variables}.
        \item The \emph{cluster algebra}~$\A(\bu, Q)$ is the subring of the \emph{ambient field}~$\F$ generated by all cluster variables, where $\F  = \Q(x_1, \ldots, x_n)$.
        \item The integer~$n$ is called the \emph{rank of the cluster algebra~$\A(\bu, Q)$}.
        \item A \emph{cluster monomial} is a product of cluster variables belonging to the same cluster.
        \item The \emph{exchange graph} of~$\A(\bu, Q)$ is the graph whose vertices are isomorphism classes of seeds of~$\A(\bu, Q)$ and where two vertices are joined by an edge if one is obtained by the other by applying one mutation.
    \end{itemize}
\end{definition}

\begin{example}
If the quiver~$Q$ has one vertex and no arrows, then the cluster algebra~$\A(\bu, Q)$ is~$\Z[u, \frac{2}{u}]$.  Its exchange graph has two vertices joined by an edge.
\end{example}

Let us recollect some of the most important results on cluster algebras.

\goodbreak

\begin{theorem}\label{theo::recollectionsCluster} 
$\phantom{x}$
    \begin{itemize}
        \item (Laurent Phenomenon) \cite{FZ02} A cluster algebra~$\A(\bu, Q)$ is contained in the Laurent polynomial ring~$\Z[u_1^{\pm 1}, \ldots, u_n^{\pm 1}]$.
              Equivalently, every cluster variable of~$\A(\bu, Q)$ is a Laurent polynomial with integer coefficients in the cluster variables of any given cluster.
              
        \item (Positivity) \cite{LS15, GHKK18} Every cluster variable of~$\A(\bu, Q)$ is a Laurent polynomial with \emph{non-negative} integer coefficients in the cluster variables of any given cluster.
        
        \item (Finite type) \cite{FZ03} A cluster algebra~$\A(\bu, Q)$ has only finitely many cluster variables if and only if~$Q$ is mutation-equivalent to a disjoint union of quivers which are orientations of Dynkin diagrams of type~$ADE$.
        
        \item (Linear independence of cluster monomials) \cite{CKLP13, GHKK18, CaoLi} Cluster monomials are linearly independent over~$\Z$.
    \end{itemize}
\end{theorem}

We shall also need the following result on the structure of non-initial variables.

\begin{lemma}[Lemma 3.7 of \cite{CKLP13}]
    Let~$\A(\bu, Q)$ be a cluster algebra.  If a cluster variable in~$\A(\bu, Q)$ is a Laurent monomial in~$\bu$ with coefficient~$\pm 1$, then it belongs to~$\bu$.
\end{lemma}

%%%%%%%%%%%%%%%%%%%%%%%%%%

\subsection{Unistructurality}

Different cluster algebras can be isomorphic as rings without having the same exchange graphs.
%; for example, \cite[Theorem 1.20]{BFZ05} shows that cluster algebras of rank~$n$ defined from acyclic quivers are isomorphic as rings to a polynomial ring in~$2n$ variables.
A stronger notion of morphism and isomorphism of cluster algebras was thus introduced and studied in \cite{ASS12, ADS14}. 
In this paper, we shall not use the full scope of this formalism, but merely the following notion from \cite{ASS14}.

\begin{definition}[Section 5.1 of \cite{ASS14}]\label{defi::unistructurality}
  Let~$(\bu, Q)$ be a seed, and let~$\A(\bu, Q) \subseteq \Q(x_1, \ldots, x_n)$ be the corresponding cluster algebra.  Denote by~$\X$ the set of its cluster variables. 
  The cluster algebra~$\A(\bu, Q)$ is \emph{unistructural} if, for any~$n$-tuple~$\bv$ which forms a free generating set of~$\Q(x_1, \ldots, x_n)$, the following condition is satisfied:
  
  \begin{itemize}
      \item [(U)] if~$(\bv, R)$ is a seed such that the set of cluster variables of~$\A(\bv, R)$ is equal to~$\X$, then there is an isomorphism between the exchange graphs of~$\A(\bu, Q)$ and~$\A(\bv, R)$ and the two cluster algebras have the same clusters (up to permutation of the cluster variables in each cluster).
  \end{itemize}
\end{definition}

In our setting, we can get rid of the condition on the exchange graphs.  

\begin{proposition}\label{prop::clustersAreSufficient}
 A cluster algebra~$\A(\bu, Q)$ with set~$\X$ of cluster variables is unistructural if and only if the following condition is satisfied:
 \begin{itemize}
      \item [(U')] if~$(\bv, R)$ is a seed such that the set of cluster variables of~$\A(\bv, R)$ is equal to~$\X$, then the two cluster algebras have the same clusters (up to permutation of the cluster variables in each cluster).
  \end{itemize}
\end{proposition}
\begin{proof}
Clearly, condition (U) implies (U').  To prove the converse, one uses the following results:
    \begin{enumerate}
        \item \cite[Theorem 4]{GSV08} for a (skew-symmetric) cluster algebra (with trivial coefficients), the same cluster (up to permutation of its variables) cannot appear in two different seeds (up to isomorphism);
        \item \cite[Theorem 5]{GSV08} in the same setting, two clusters (up to permutation) belong to adjacent seeds of the exchange graph if and only if they have exactly~$n-1$ cluster variables in common. 
    \end{enumerate}
    Therefore, if the set of clusters of a cluster algebra is known, then its exchange graph can be recovered using the adjacency relations in (2).  
    Thus two cluster algebras having the same set of clusters must have isomorphic exchange graphs.
\end{proof}

\begin{remark}
 Proposition~\ref{prop::clustersAreSufficient} is true in a greater generality, since the two results it uses, \cite[Theorems 4 and 5]{GSV08}, are proved under the following conditions. Let~$B$ be a skew-symmetrizable matrix, and~$\by$ be a set of coefficients in a semi-field.  Then \cite[Theorems 4 and 5]{GSV08} are true for the cluster algebra with coefficients~$\A(\bu, \by, B)$ if either
 \begin{itemize}
     \item the cluster algebra is of geometric type, or
     \item $B$ is of full rank.
 \end{itemize}
 Since all cluster algebras in this paper are of geometric type (they have trivial coefficients), we will not say more about this.
\end{remark}

\begin{remark}\label{rema::quivers or opposite}
    In the definition of unistructurality, we can also say something about seeds: a cluster algebra~$\A(\bu, Q)$ is unistructural if and only if the following condition is satisfied:
    
    \begin{itemize}
        \item[(U'')] if~$(\bv, R)$ is a seed such that the set of cluster variables of~$\A(\bv, R)$ is equal to~$\X$, and if~$R$ is a disjoint union of quivers~$R^1, \ldots, R^r$,  then there is a seed~$(\bv , Q^1 \sqcup \ldots \sqcup Q^r )$ which is mutation-equivalent to~$(\bu, Q)$ (up to isomorphism), where each~$Q^j$ is either~$R^j$ or the opposite quiver~$(R^j)^{op}$.
    \end{itemize}
    To see this, notice that at the end of the proof of Proposition~\ref{prop::clustersAreSufficient}, the adjacent clusters in the exchange graph also allow one to recover the exchange relations, and thus to deduce the quiver of the associated seed up to a change in the orientation of all the arrows in each of its connected components.  Hence, condition~(U') implies~(U'').  The converse holds trivially, so the equivalence is proved.
\end{remark}

%\begin{remark}
% If~$Q$ is a connected quiver, then the definition of unistructurality is slightly simpler: the cluster algebra~$\A(\bu, Q)$ is unistructural if and only if the following condition is satisfied:
%   \begin{itemize}
%      \item[(Uc)] if~$(\bu', Q')$ is a seed such that the set of cluster variables of~$\A(\bu', Q')$ is equal to~$\X$, then~$(\bu', Q')$ or~$(\bu', (Q')^{op})$ is mutation-equivalent to~$(\bu, Q)$ (up to isomorphism).
%  \end{itemize}
%\end{remark}

\begin{example}
A cluster algebra of rank~$1$ is unistructural, since it only has two clusters and two cluster variables.
\end{example}

The above example is one of the few cluster algebras which are known to be unistructural.

\begin{theorem} The following results are known for skew-symmetric cluster algebras with trivial coefficients.
    \begin{enumerate}
        \item \cite[Theorem 5.2]{ASS14} Cluster algebras of rank~$2$ are unistructural.
        \item \cite[Theorem 0.1]{ASS14-2} Cluster algebras of Dynkin type are unistructural.
        \item \cite[Theorem 3.2]{BM16} Cluster algebras of affine type~$\tilde{\mathbb{A}}_n$ are unistructural.
    \end{enumerate}
\end{theorem}

Moreover, there is no known counter-example to the following conjecture.

\begin{conjecture} \cite[Conjecture 1.2]{ASS14}\label{conj::unistructurality}
Any cluster algebra is unistructural.
\end{conjecture}

%%%%%%%%%%%%%%%%%%%%%%%%%%

\subsection{Reduction to connected quivers}

In this section, we show that it is enough to consider connected quivers when dealing with unistructurality.  More precisely, we prove the following proposition.

\begin{proposition}\label{prop::connected}
    Let~$Q$ be a quiver which is a disjoint union of quivers~$Q^1, \ldots, Q^r$.  Let~$(\bu, Q)$ be a seed, and let~$(\bu^1, Q^1), \ldots, (\bu^r, Q^r)$ be the seeds corresponding to the decomposition of~$Q$.  Then the cluster algebra~$\A(\bu, Q)$ is unistructural if and only if the cluster algebra~$\A(\bu^i, Q^i)$ is unistructural for all~$i\in \{1, \ldots, r\}$.
\end{proposition}
Before proving Proposition \ref{prop::connected}, let us state an immediate consequence.

\begin{crl}
Conjecture~\ref{conj::unistructurality} is true if and only if for any \emph{connected} quiver~$Q$ and any seed~$(\bu, Q)$, the cluster algebra~$\A(\bu, Q)$ is unistructural.
\end{crl}

Let us now prove Proposition~\ref{prop::connected}.  Assume that~$Q$ is a quiver which is a disjoint union of quivers~$Q^1, \ldots, Q^r$, and let~$(\bu, Q),(\bu^1, Q^1), \ldots, (\bu^r, Q^r)$ be seeds as above.
One implication of the proposition is clear.

\begin{lemma}
    If the cluster algebra~$\A(\bu, Q)$ is unistructural, then so is~$\A(\bu^i, Q^i)$ for each~$i\in \{1, \ldots, r\}$.
\end{lemma}
\begin{proof}
    Suppose that~$\A(\bu^i, Q^i)$ is not unistructural for a certain~$i$.  Let $(\bv^i, R^i)$ be a seed such that~$\A(\bv^i, R^i)$ has the same cluster variables as~$\A(\bu^i, Q^i)$, but such that~$(\bv^i, R^i)$ has condition (U) in Definition \ref{defi::unistructurality} fail.
    
    Let~$\bu' = \bu^1 \sqcup \ldots \sqcup \bu^{i-1} \sqcup \bv^i \sqcup \bu^{i+1} \sqcup \ldots \sqcup \bu^r$ and~$Q' = Q^1 \sqcup \ldots \sqcup Q^{i-1} \sqcup R^i \sqcup Q^{i+1} \sqcup \ldots \sqcup Q^r$.  Then~$\A(\bu, Q)$ and~$\A(\bu', Q')$ have the same cluster variables, but condition (U) fails.  Thus~$\A(\bu, Q)$ is not unistructural.
\end{proof}

Assume now that each~$\A(\bu^i, Q^i)$ is unistructural.  
Note that each cluster variable of~$\A(\bu, Q)$ lies in one of the fields~$\Q(\bv^i)$, for some~$i\in \{1, \ldots, r\}$.
Let~$(\bv, R)$ be a seed such that~$\A(\bv, R)$ and~$\A(\bu, Q)$ have the same cluster variables.

\begin{lemma}\label{lemm::disjointUnion}
    The quiver~$R$ is a disjoint union of quivers~$R^1, \ldots, R^r$ and~$\bv$ is a corresponding disjoint union of tuples~$\bv^1, \ldots, \bv^r$ in such a way that each cluster algebra~$\A(\bv^i, R^i)$ is contained in~$\Q(\bu^i)$.
\end{lemma}
\begin{proof}
    By assumption, each element of~$\bv$ is contained in one of the subfields~$\Q(\bu^i)$ of~$\Q(\bu)$.  Since $\bv$ is a free generating set of~$\Q(\bu)$, all subsets of~$\bv$ are algebraically free.  Therefore, for each~$i$, there is exactly~$|\bu^i|$ elements of~$\bv$ in~$\Q(\bu^i)$.
    This allows us to partition~$\bv$ into sub-tuples~$\bv^1, \ldots, \bv^r$ in such a way that~$\bv^i$ contains precisely the variables of~$\bv$ contained in~$\Q(\bu^i)$.
    Since~$\bv$ freely generates~$\Q(\bu)$, this implies that each~$\bv^i$ freely generates~$\Q(\bu^i)$; in other words,~$\Q(\bu^i)= \Q(\bv^i)$ for each~$i$.

    Now, let~$v_j$ a cluster variable of~$\bv$, and let~$v'_j$ be the variable obtained by mutation of~$(\bv, R)$ at~$j$.  By the above argument, if~$v_j$ is in~$\Q(\bu^i)$, then so is~$v'_j$.
    Write the exchange relation as
    \[
      v_jv'_j = \prod_{\stackrel{\alpha \in R_1}{t(\alpha) = j}} v_{s(\alpha)} + \prod_{\stackrel{\alpha \in R_1}{s(\alpha) = j}} v_{t(\alpha)}.
    \]
    Since~$v_jv'_j$ lies in~$\Q(\bu^i) = \Q(\bv^i)$ and each~$v_k$ lies in~$\Q(\bv)$, we must have that all variables occurring on the right hand side of the exchange relation are also in~$\Q(\bv^i)$.  Thus the vertex~$j$ of~$R$ is related by arrows only to vertices corresponding to variables in~$\bv^i$.
    Repeating the argument for all vertices~$j$ of~$R$, we partition~$R$ into a disjoint union of quivers~$R^1, \ldots, R^r$ as required.
\end{proof}

As a consequence of Lemma~\ref{lemm::disjointUnion}, the set of cluster variables of~$\A(\bv^i, R^i)$ is the same as that of~$\A(\bu^i, Q^i)$.  Since the latter is unistructural by hypothesis, we get that the two algebras have the same clusters.  As~$R$ is the disjoint union of the~$R^i$ and~$\bv$ that of the~$\bv^i$, we thus get that~$\A(\bu, Q)$ and~$\A(\bv, R)$ also have the same clusters.  Thus~$\A(\bu, Q)$ is unistructural.  This finishes the proof of Proposition~\ref{prop::connected}.

%%%%%%%%%%%%%%%%%%%%%%%%%%

\section{Cluster algebras arising from triangulations of surfaces}\label{sect::surfaces}

%%%%%%%%%%%%%%%%%%%%%%%%%%

\subsection{Cluster algebras and surfaces}

We now recall how a cluster algebra can be associated to a triangulation of a surface.  We mainly follow \cite{FST08}, but restrict ourselves to the setting of unpunctured surfaces.

A \emph{marked surface} is a pair~$(S,M)$, where~$S$ is a connected oriented Riemann surface with boundary~$\partial S$ and~$M$ is a finite subset of~$\partial S$ such that~$M$ has at least one point on each connected component of~$\partial S$.  We exclude the case where~$(S,M)$ is such that~$S$ is a disk and~$|M| \in \{1,2, 3\}$; these are the cases where there is either only one triangulation or no triangulation at all (in the sense that we recall below).

An \emph{arc} on~$(S,M)$ is an isotopy class of curves on~$S$ with endpoints in~$M$.  A \emph{boundary arc} is an arc isotopic to a curve contained in the boundary~$\partial S$ of~$S$; an \emph{internal arc} is an arc which is not a boundary arc.  Finally, a \emph{closed curve} on~$(S,M)$ is a free isotopy class of curves on~$S$ whose starting point and ending point are the same point in the interior of~$S$ (a \emph{free isotopy} is an isotopy of curves that does not necessarily fix their endpoints).

The figure below represents two triangulated surfaces: on the left, a sphere with two open disks removed (or annulus), and on the right, a sphere with three open disks removed.

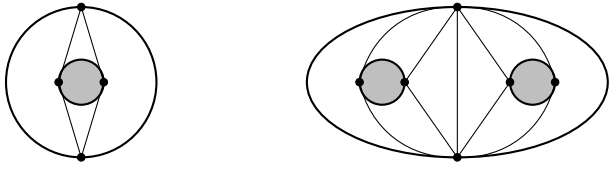
\begin{figure}[h!]

    \begin{displaymath}
		\begin{tikzpicture}
		  % First surface: Two boundary components
		  \draw[thick] (0,0) circle (1);
		  \filldraw[grey] (0,0) circle (0.3);
		  \draw[thick] (0,0) circle (0.3);
		  
		  % Marked points
		  \filldraw[black] (0.3, 0) circle (0.05);
		  \filldraw[black] (-0.3, 0) circle (0.05);
		  \filldraw[black] (0, 1) circle (0.05);
		  \filldraw[black] (0, -1) circle (0.05);
		  
		  % Arcs
		  \draw (0.3, 0) -- (0,1);
		  \draw (-0.3, 0) -- (0,1);
		  \draw (0.3, 0) -- (0,-1);
		  \draw (-0.3, 0) -- (0,-1);
		  
		  % Second surface: Three boundary components
		  \draw[thick] (5,0) ellipse (2cm and 1cm);
		  \filldraw[grey] (4,0) circle (0.3);
		  \draw[thick] (4,0) circle (0.3);
		  \filldraw[grey] (6,0) circle (0.3);
		  \draw[thick] (6,0) circle (0.3);
		  
		  % Marked points
		  \filldraw[black] (4.3, 0) circle (0.05);
		  \filldraw[black] (3.7, 0) circle (0.05);
		  \filldraw[black] (6.3, 0) circle (0.05);
		  \filldraw[black] (5.7, 0) circle (0.05);
		  \filldraw[black] (5,1) circle (0.05);
		  \filldraw[black] (5,-1) circle (0.05);
		  
		  % Arcs
		  \draw (5,1) -- (5,-1);
		  \draw (5,1) -- (5.7, 0);
		  \draw (5,1) -- (4.3, 0);
		  \draw (5,-1) -- (5.7, 0);
		  \draw (5,-1) -- (4.3, 0);
		  \draw (5,-1) to [bend left=40] (3.7,0);
		  \draw (5,-1) to [bend right=40] (6.3,0);
		  \draw (5,1) to [bend right=40] (3.7,0);
		  \draw (5,1) to [bend left=40] (6.3,0);

		\end{tikzpicture}
	\end{displaymath}  
\end{figure}

We say that two arcs or closed curves~$\gamma$ and~$\delta$ on~$(S,M)$ \emph{intersect} if, for all choices of isotopy representatives~$\bar \gamma$ of~$\gamma$ and~$\bar \delta$ of~$\delta$, the two representatives~$\bar\gamma$ and~$\bar\delta$ intersect in the interior of~$S$.  We say that an arc or closed curve~\emph{self-intersects} if all isotopy representatives of it intersects itself in the interior of~$S$.  A \emph{triangulation} of~$(S,M)$ is a maximal collection of non-self-intersecting and pairwise non-intersecting curves.

For a triangulation~$\tau$ and an arc~$i$ in~$\tau$, the \emph{flip of~$\tau$ at~$i$} is the unique triangulation~$\mathfrak{f}_i(\tau)$ containing~$\tau\setminus i$ but not~$i$.

Let~$\tau$ be a triangulation of~$(S,M)$.  Define a quiver~$Q(\tau)$ in the following way:
\begin{itemize}
    \item vertices of~$Q(\tau)$ are arcs in~$\tau$;
    \item there is an arrow $i\to j$ for each triangle~$\Delta$ of~$\tau$ in which both~$i$ and~$j$ occur, with~$j$ immediately following~$i$ in the clockwise order of the boundary of~$\Delta$.
\end{itemize}

\begin{definition}
  Let~$\tau$ be a triangulation of~$(S,M)$ with~$n$ arcs, and let~$\bu$ be a free generating set of~$\Q(x_1, \ldots, x_n)$. The~\emph{cluster algebra associated with~$\tau$} (with set~$\bu$ of initial variables) is the cluster algebra~$\A(\bu, \tau) := \A(\bu, Q(\tau))$.  
\end{definition}

Let~$\tau$ be a triangulation of a marked surface~$(S,M)$.  The following are consequences of \cite[Theorem 7.11]{FST08}.

\begin{theorem}[Theorem 7.11 of \cite{FST08}]
$\phantom{x}$
\begin{itemize}
    \item  The cluster variables of~$\A(\bu, \tau)$ are in bijection with non-self-intersecting internal arcs of~$(S,M)$.  If~$\gamma$ is such an arc, denote by~$u_{\gamma}$ the corresponding cluster variable.  Under this bijection, the arcs of~$\tau$ are sent to the cluster variables in $\bu$.
    \item  The above induces a bijection between triangulations of~$(S,M)$ and seeds of~$\A(\bu, \tau)$:
    if~$\tau'$ is a triangulation of~$(S,M)$, then the associated seed is~$\big( (u_{\gamma} \ | \ \gamma\in \tau'), Q(\tau')  \big)$.
    \item In particular, if~$\tau'$ and~$\tau''$ are related by a flip at~$i$, then the seed associated to~$\tau''$ is the mutation at~$i$ of the seed associated to~$\tau'$.
\end{itemize}
\end{theorem}

%%%%%%%%%%%%%%%%%%%%%%%%%%

\subsection{Bracelets}

The following definitions are adapted from \cite{MSW13}.

\begin{definition}
    Let $\g$ be an essential loop in $(S,M)$. The \df{bracelet} $\Brac_m \g$ is the closed loop obtained by concatenating $\g$ with itself exactly $m$ times. We denote the polynomial in $\F$ associated to it by $b_m(\g)$.
\end{definition}

Note that $\Brac_m \g$ has $m-1$ self-intersections.

%We know the following property from Proposition 4.4 in \cite{Thu14}.

\begin{lemma}[Theorem 1.1 of \cite{MW13}]\label{LaurentBracelet}
     Every bracelet has a positive Laurent expansion with respect to any cluster of $\mathcal{A}(\bu,\tau).$
\end{lemma}

\begin{comment}
\begin{proposition}\label{bracelet}
    The polynomial in $\F$ associated to a bracelet is clusterisable.
\end{proposition}
    
\begin{proof}
    Let $B=\Brac_m\g$ be a bracelet and $b=b_m(\g)$. A COMPLETET
\end{proof}
\end{comment}

%%%%%%%%%%%%%%%%%%%%%%%%%%
%\subsection{The bracelets basis}

\begin{definition}
    A collection $c$ of arcs is \df{$\C$-compatible} if it contains no intersections (between two elements or self-intersections). Therefore, it excludes bracelets.
    
    A collection $c'$ of arcs and bracelets is \df{$\C'$-compatible} if:
    \begin{itemize}
        \item no two elements of $c'$ intersect each other, except for the self-intersections of a bracelet;
        % \item given a essential loop $\g$ in $(S,M)$, there at most  one $m \geq 1$ such that the $m$-th bracelet lies in $C_2$, and, moreover, there is at most one copy of this bracelet $\Brac_m(\g)$ in $C_2$.
        \item there exists at least one essential loop $\g$ in $(S,M)$ such that the bracelet $\Brac_m(\g)$ lies in $c'$, with $m \geq 1$. Moreover, there is exactly one copy of $\Brac_m(\g)$ in $c'$ and if $m' \neq m$, then $\Brac_{m'}(\g) \not\in c'$.
    \end{itemize}
    \end{definition}
    
In other word, there are no bracelets in $\C$-compatible collections of arcs, but there are always bracelets in $\C'$-compatible collections of arcs and at most one per essential loop. 
Note that it does not completely agree with the definition if $\C$-compatibility from \cite{MSW13}: a collection of arcs is $\C$-compatible in \cite{MSW13} if it is $\C$-compatible or $\C'$-compatible.

\begin{comment}
    \begin{definition}
        A collection $C$ of arcs \df{$\C$-compatible} no two elements of $C$ intersect each other.
        
        A collection $C'$ of arcs and bracelets is \df{$\C'$-compatible} if:
        \begin{itemize}
            \item no two elements of $C'$ intersect each other, except for the self-intersection of a bracelet
            \item given a essential loop $\g$ in $(S,M)$, there is at most one $m \geq 1$ such that the $m$-th bracelet lies in $C$, and, moreover, there is at most one copy of this bracelet $\Brac_m(\g)$ in $C'$.
        \end{itemize}
    \end{definition}
\end{comment}

We define $\C (S,M)$ to be the set of all $\C$-compatible collections of arcs in $(S,M)$ and we define $\C'(S,M)$ to be the set of all $\C'$-compatible collections of arcs and bracelet in $(S,M)$. collections.
We denote
\[ \B = \left\{ \prod_{\g \in C} x_\g \mid c \in \C (S,M)  \right\}
\text{ and }
\B' = \left\{ \prod_{\g \in C'} x_\g \mid c' \in \C' (S,M)  \right\}.\]

\begin{definition}
    Let $A$ be a $\Z$-algebra free as a $\Z$-module and $B$ be a basis of $A$.
    The basis $B$ is \df{positive}, if for every $b_1$, $b_2 \in B$,
    \[b_1 b_2 = \sum_{b \in B} m_b b \]
    with $m_b \geq 0$ for every $b \in B$.
\end{definition}

\begin{theorem}[\cite{MSW13,Thu14,CLS}]\label{positivebasis}
    The set $\B \cup \B'$ is a positive basis for the cluster algebra~$\A(\bu, \tau)$.
\end{theorem}

It is proved in \cite[Theorem 1.1]{MSW13,CLS} that~$\B \cup \B'$ is a basis and in \cite[Theorem 1]{Thu14} that it is positive.

%%%%%%%%%%%%%%%%%%%%%%%%%%

\subsection{Products of cluster variables and skein relations}

We recall here some definitions and results from \cite{MW13}. An \emph{multicurve} a multiset of arcs and closed curves on the surface.

\begin{definition} [Definition 6.1 of \cite{MW13}]
    Let $M$ be a multicurve such that we have one of the following two cases:
    \begin{itemize}
        \item $M = \{ \gamma_1, \gamma_2\}$ where $\gamma_1$ and $\gamma_2$ are arcs or closed curves intersecting at a point~$p$;
        \item $M = \{ \gamma \}$ where $\gamma$ has a self-intersection at a point $p$.
    \end{itemize}
    The \emph{smoothing} of $M$ at point $p$ is $s_p(M) = \{ M_1, M_2\}$ where $M_1$ (respectively $M_2$) is the same as $M$ except for the local change at $p$ that replaces the intersection or self-intersection point $\times$ with the pair of curves segments {\Large$~_\cap^{\cup}$} (or {$\supset \subset$}, respectively).
\end{definition}

The skein relations are an important theorem from \cite{MW13} and are recalled in the following statement.

\begin{theorem}[Skein relations, \cite{MW13}] \label{smoothing}
    Let $m$ be a multicurve with an intersection at point $p$ and let $s_p(m) = \{ M_1, M_2\}$. 
    Denote by $x_C$, $x_1$ and $x_2$ the product of the polynomials in $\F$ associated to the curves in $C$, $M_1$ and $M_2$, respectively. Then,
    \[ x_C = \pm x_1 \pm x_2. \]
\end{theorem}

\begin{lemma}\label{noclosedloops}
    Let $\gamma_1, ..., \gamma_r$ be arcs on the surface.  Then the smooth resolution of the multicurve $\{\gamma_1, ..., \gamma_r\}$ contains at least one multicurve with exactly $r$ arcs and no closed loops.
\end{lemma}
\begin{proof}
    We prove the result by induction on the number $n$ of intersection points of the multicurve $\{\gamma_1, ..., \gamma_r\}$.
    
    Let $p$ be an intersection point.  Then either $p$ is an point of intersection between two distinct arcs $\gamma_i$ and $\gamma_j$, or $p$ is a point of self-intersection of an arc $\gamma_k$.
    
    In the first case, smoothing at $p$ will give rise to two multicurves, each made of exactly $r$ arcs and with at most $n-1$ intersection points.
    Locally, the situation is as in the following picture, where the two arcs intersecting at $P$ are depicted around $p$.  In the picture, we draw a multicurve $\{\gamma_1, \ldots, \gamma_r\}$ instead of writing the product $x_{\gamma_1} \cdots x_{\gamma_r}$.
    
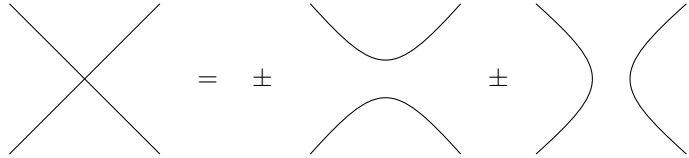
    \begin{displaymath}
		\begin{tikzpicture}
		
		\draw (-1,1) -- (1,-1); % A line of the crossing.
		\draw (-1,-1) -- (1,1); % Another line of the crossing.
		
		\draw (2,0) node {$= \quad \pm$}; % Equality.
		
		\draw (3,1) .. controls (3.9,0) and (4.1,0) .. (5,1); % A line of the first resolution.
		\draw (3,-1) .. controls (3.9,0) and (4.1,0) .. (5,-1); % A line of the first resolution.
		
		\draw (5.5,0) node {$\pm$}; % Plus sign.
		
		\draw (6,1) .. controls (7,0.1) and (7,-0.1) .. (6,-1); % A line of the second resolution.
		\draw (8,1) .. controls (7,0.1) and (7,-0.1) .. (8,-1); % A line of the second resolution.
		\end{tikzpicture}
	\end{displaymath}  
    
    In the second case, smoothing at $p$ gives rise to two multicurves: one is made of $r$ arcs and a closed curve, and the other is made of $r$ arcs.  Both multicurves have at most $n-1$ intersection points.
    This is illustrated in the following picture.
    Note that the loop in the leftmost and rightmost terms may not be contractible.
    
    
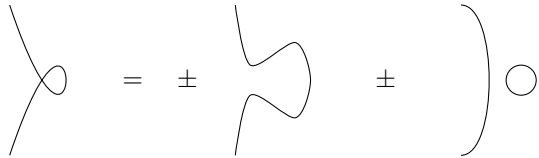
\begin{displaymath}
		\begin{tikzpicture}
		
		\draw (0,1) .. controls (1,-2) and (1,2) .. (0,-1); % A self-intersecting curve.

		\draw (2,0) node {$= \quad \pm$}; % Equality.
		
		\draw plot [smooth] coordinates {(3,1) (3.2,0.2) (3.8,0.5) (4,0) (3.8,-0.5) (3.2,-0.2) (3,-1)}; % The unique curve in the first resolution.
		
		\draw (5,0) node {$\pm$}; % Plus sign.
		
		\draw (6,1) .. controls (6.5,1) and (6.5,-1) .. (6,-1); % A line of the second resolution.
		\draw (6.8,0) circle (0.2); % A circle of the second resolution.
		\end{tikzpicture}
	\end{displaymath} 
    
    In each case, one of the multicurves obtained by smoothing at $p$ is made only of arcs. This finishes the proof.
\end{proof}

\begin{lemma} \label{termeva}
    Let $\A(\bx,Q)$ be a cluster algebra associated to a triangulation of a surface. Let $x_1$ and $x_2$ be cluster variables. Then
    \[ x_1x_2 = \sum_{b \in \B} m_b b + \sum_{b' \in \B'} m_{b'} b'\]
    with $m_b$, $m_{b'} \in \N$ for every $b \in B$, $b' \in B'$, and
    \[ \sum_{b \in B} m_b > 0. \]
\end{lemma}

\begin{proof}
    This follows from Theorem \ref{positivebasis} and Lemma \ref{noclosedloops}.
\end{proof}

%%%%%%%%%%%%%%%%%%%%%%%%%%

\section{Unistructurality of cluster algebra from unpunctured surface}\label{sect::proof}

%%%%%%%%%%%%%%%%%%%%%%%%%%

\subsection{Proof of the main result}

\begin{theorem}\label{theo::main}
    Cluster algebras with trivial coefficients from marked surfaces without punctures are unistructural.
\end{theorem}

\begin{proof}
    Let $\A(\bu,Q)$ be a cluster algebra associated to a triangulation of a marked surface $(S,M)$ without punctures, where $\bu=\{u_1,\dots,u_{n}\}$.
    Denote by~$\X$ the set of cluster variables of~$\A(\bu, Q)$.
    Let $(\bv,R)$ be a seed such that the cluster variables of $\mathscr{A}(\bu,Q)$ and of $\mathscr{A}(\bv,R)$ are the same; denote by~$\Y$ the set of cluster variables of~$\mathscr{A}(\bv,R)$.
    Note that a consequence of that hypothesis is that $\bu$ and $\bv$ have the same number of elements, which is the cardinality of a transcendence basis of the (common) ambient field. We show that $\bv$ is a cluster of $\mathscr{A}(\bu,Q)$ by using a proof very similar to the one of Lemma 3.1 in \cite{BM16}.
    
    On the contrary, suppose that there exist $v_1$ and $v_2$ in $\X=\Y$ which are compatible in $\A(\bv,R)$, but not in $\A(\bu,Q)$. Without loss of generality, we can furthermore assume that $v_1$ and $v_2$ belong to $\bv$.
    Since $v_1$ and $v_2$ are not compatible in $\A(\bu,Q)$, they are associated to arcs that intersect each other in $(S,M)$.
    We know from Lemma \ref{termeva} that
    \begin{equation} \label{produitvar}
        v_1v_2 = \sum_{b \in \B} m_b b + \sum_{b' \in \B'} m'_b b',
    \end{equation}
    with $\sum_{b \in \B} m_b b \neq 0$.
    
    Using Lemma \ref{LaurentBracelet}, we know that each $b$ is a cluster monomial and each $b'$ is a Laurent polynomial in cluster variables with positive coefficients.
    Since $\bv$ is a cluster, each $b \in \B$ can be written as Laurent polynomial in $\bv$. 
    The same is true for $\sum_{b \in \B} m_b b$, so write
    $$
      \sum_{b \in \B} m_b b = \frac{p_{\bv}}{m_{\bv}},
    $$
    where $p_{\bv}$ (respectively $m_{\bv}$) is a polynomial (respectively a monomial) in $\bv$.
    
    If $\sum_{b' \in \B'} m'_b b' = 0$, we obtain that $$v_1v_2m_{\bv} = p_{\bv},$$ 
    so that~$p_{\bv}$ is a monomial in~$\bv$.
    Now, each~$b$ in the sum~$\sum_{b \in \B} m_b b$ is a product of cluster variables in~$\X=\Y$; it is thus a Laurent polynomial in~$\bv$ with positive coefficients by Theorem~\ref{theo::recollectionsCluster}.  Moreover, the only cluster variables in~$\Y$ which are Laurent monomials in~$\bv$ are the elements of~$\bv$ itself by \cite[Lemma 3.7]{CKLP13}.  
    Thus, if~$p_{\bv}$ is a monomial in~$\bv$, then the sum~$\sum_{b \in \B} m_b b$ has only one term, say~$b$, which has to be a Laurent monomial in~$\bv$; since~$b$ is a product of elements of~$\X = \Y$, it is thus a monomial in~$\bv$.
    By the linear independence of cluster monomials (see, in this generality, \cite{CL12}), this implies that~$b=v_1v_2$.  This is then the decomposition of~$b$ as a product of compatible elements of~$\X$; therefore,~$v_1$ and~$v_2$ are compatible in~$\A(\bu, Q)$, a contradiction.

    %which is a contradiction to the algebraic independence of $\by$.
    
    Now, suppose that $\sum_{b' \in \B'} m'_b b' \neq 0$.
    Because $\sum_{b' \in \B'} m'_b b'$ is a sum of Laurent polynomials in cluster variables with positive coefficients, we can write 
    $$ \sum_{b' \in \B'} m'_b b' = \frac{p_\X}{m_\X} $$
    where $p_\X$ (respectively $m_\X$) is a polynomial (respectively a monomial) with positive coefficients in $\X$.
    Now, as before, we change notation to write $p_\X$ and $m_\X$ as Laurent polynomials in $\bv$ with positive coefficients:
    $$ p_\X = \frac{p'_{\bv}}{m'_{\bv}} \text{ and } m_\X = \frac{p''_{\bv}}{m''_{\bv}}.$$
    Equation \ref{produitvar} becomes
    $$ v_1 v_2  = \frac{p_{\bv}}{m_{\bv}} + \frac{\left(\frac{p'_{\bv}}{m'_{\bv}}\right)}{\left(\frac{p''_{\bv}}{m''_{\bv}}\right)},$$
    which is equivalent to 
    $$ m_{\bv}m''_{\bv}p'_{\bv} = m'_{\bv}p''_{\bv}(y_1y_2m_{\bv} - p_{\bv}). $$
    Note that in the left-hand sight of the equation, there is a polynomial in $\bv$ with positive coefficients, while there is a minus sign on the right-hand sight.
    Since $x_1x_2m'_y$ is a monomial, we deduce that $y_1y_2m'_{\bv} = p_{\bv}$, and we obtain as above that~$y_1$ and~$y_2$ are compatible in~$\A(\bu, Q)$, a contradiction.
    
    %which is also a contradiction to the algebraic independence of $\by$.
    
    Therefore, $v_1$ and $v_2$ are two compatible variables in $\A(\bu,Q)$ if they are compatible in $\A(\bv,R)$.
    
    It follows from the above that any cluster of~$\A(\bv,R)$ is a cluster of~$\A(\bu,R)$.  This also implies the converse: any cluster of~$\A(\bu,R)$ is a cluster of~$\A(\bv,R)$.  Indeed, let~$\bw$ be a seed in~$\A(\bv,R)$, and let~$\bw_1$, \ldots, $\bw_n$ be the neighbouring clusters in the exchange graph of~$\A(\bv,R)$.  Then, by~\cite[Theorem 5]{GSV08}, $\bw, \bw_1, \ldots, \bw_n$ are neighbouring clusters in~$\A(\bu,R)$, since for every~$i$, $\bw$ and~$\bw_i$ only differ by one cluster variable.   As in Remark~\ref{rema::quivers or opposite}, this allows us to recover the quiver~$Q$ in the seed~$(\bw,Q)$ containing~$\bw$, up to a change of orientation of each of its connected components.  This argument works in~$\A(\bu,R)$ as well as in~$\A(\bv,R)$, so we get a seed~$(\bw,Q')$ in~$\A(\bu,R)$.  Since the connected components of~$Q$ and~$Q'$ are the same up to a change of orientation, the cluster algebras~$\A(\bu,R)$ and~$\A(\bv,R)$ have the same clusters.  By Proposition~\ref{prop::clustersAreSufficient}, this shows that~$\A(\bu,Q)$ is unistructural.

\end{proof}

\subsection{Consequences for cluster automorphisms}
The idea of unistructurality first appeared in \cite{ASS14} while studying automorphisms between cluster algebras.

\begin{definition}[\cite{ASS12}]
    Let $\A(\bu,Q)$ be a cluster algebra, and let $f : \A(\bu,Q) \rightarrow \A(\bu,Q)$ be an automorphism of $\Z$-algebras.
    Then $f$ is called a \emph{cluster automorphism} if there exists a seed $(\bv,R)$ of $\A(\bu,Q)$, such that
    the following conditions are satisfied: \begin{itemize}
    \item $f(\bv)$ is a cluster in $\A(\bu,Q)$;
    \item $f$ is compatible with mutations, that is, for every $v \in \bv$, we have \[f\left(\mu_{v}(\bv)\right) = \mu_{f(v)}(f(\bv)).\]
    \end{itemize}
    \end{definition}
    
    Remark that the existence of such a cluster $\bv$ guarantees that the image of any cluster is a cluster.
    
    The following is a direct consequence of \cite[Lemma 5.2]{ASS14} and Theorem~\ref{theo::main}.
    
    \begin{crl}
    Let $\A(\bu,\tau)$ be a cluster algebra arising from a triangulation~$\tau$ of an unpunctured marked surface~$(S,M)$. Then a map $f:\A(\bu,\tau) \rightarrow \A(\bu,\tau)$ is a cluster automorphism if and only if $f$ is an automorphism of the ambient field which restricts to a permutation of the set of cluster variables.
    \end{crl}

%\begin{proof}
%Lemma 5.2 in \cite{ASS14} states that this result is true for unistructural cluster algebras. It then follows from Theorem \cite{theo::main}.
%\end{proof}

It is conjectured in \cite{ASS14} that this result is true for any cluster algebra, since, by \cite[Lemma 5.2]{ASS14}, it is true for any unistructural cluster algebra, and all cluster algebras are conjectured to be unistructural.

\section*{Acknowledgements}

The authors would like to thank Hugh Thomas for fruitful discussions and his comments on an earlier version of this work which resulted in a simplification of one of the arguments.

\bibliographystyle{alpha}
\bibliography{biblio}

\end{document}